\documentclass[12pt]{amsart}
\usepackage{amsmath,amsfonts}
\usepackage{amssymb}
\usepackage{amscd}
\usepackage{amsthm}
\usepackage{yhmath}
\usepackage{subfigure}

\usepackage[all]{xy}
\usepackage{color}
\usepackage{comment}
\usepackage{mathrsfs}

\numberwithin{equation}{section}

\newtheorem{theorem}[equation]{Theorem}

\newtheorem{proposition}[equation]{Proposition}

\newtheorem*{theorem*}{Theorem}

\newtheorem{lemma}[equation]{Lemma}

\newtheorem*{corollary*}{Corollary}

\theoremstyle{remark}

\theoremstyle{definition}

\newtheorem{example}[equation]{Example}

\def\XXint#1#2#3{{\setbox0=\hbox{$#1{#2#3}{\int}$}
	\vcenter{\hbox{$#2#3$}}\kern-.5\wd0}}

\newcommand{\Ad}{\operatorname{Ad}}
\newcommand{\ad}{\operatorname{ad}}

\newcommand{\C}{\mathbb C}

\newcommand{\R}{\mathbb R}

\newcommand{\g}{\mathfrak{g}}

\newcommand{\al}{\alpha}

\def\ga{\gamma}

\def\la{\lambda}

\def\Om{\Omega}

\def\be{\beta}

\newcommand{\End}{{\rm{End}}}
\newcommand{\Span}{{\rm span}}
\newcommand{\Abn}{{\rm{Abn}}}
\newcommand{\dd}{{\rm{d}}}

\newcommand{\f}{\mathfrak{f}}
\newcommand{\Pg}{P_\gamma}
\newcommand{\ep}{\varepsilon}

\newcommand{\IE}{\mathscr E}
\newcommand{\codim}{\operatorname{codim}}

\begin{document}

\title
{On the codimension of the abnormal set in step two Carnot groups}

\author[Ottazzi]{Alessandro Ottazzi}
\address[Ottazzi]{School of Mathematics and Statistics, UNSW Kensington campus, 2052 NSW Sydney.}
\email{a.ottazzi@unsw.edu.au}

\author[Vittone]{Davide Vittone}
\address[Vittone]{Dipartimento di Matematica, via Trieste 63, 35121 Padova, Italy.}
\email{vittone@math.unipd.it}

\thanks{{D.V. is supported by University of Padova, Project Networking, and GNAMPA of INdAM (Italy), project Campi vettoriali, superfici e perimetri in geometrie singolari.}}
\thanks{{A.O. was partially supported by the Australian Research Councils Discovery Projects funding scheme, projects no. DP140100531 and DP170103025}}
\keywords{{Sard property, endpoint map, abnormal curves, Carnot groups, sub-Riemannian geometry.}}

\subjclass[2010]{{53C17, 22E25, 14M17.}}

\begin{abstract}
In this article we prove that the codimension of the abnormal set of the endpoint map for certain  classes of  Carnot groups of step $2$
is at least three. 
Our result applies to all step 2 Carnot groups of dimension up to 7 and is a generalisation of a previous analogous result for step $2$ free nilpotent groups.
\end{abstract}


\maketitle
%

\section{Introduction}
Let $G$ be a {\em Carnot group}, i.e., a connected and simply connected Lie group with {\em stratified} nilpotent Lie algebra $\g=V_1\oplus\dots\oplus V_s$. 
Let $\End$ be the  endpoint map
\begin{eqnarray*}
\End: L^2([0,1],V_1)& \rightarrow &G\\
u\qquad&\mapsto& \gamma_u(1),
 \end{eqnarray*}
where
$ \gamma_u$ is the curve on $G$ leaving from the identity $e \in G$ with {$\dot\gamma_u(t)= (\dd L_{\gamma_u(t)} )_e  u (t)$},  $L_g$ denoting  left translation by $g$. The {\em abnormal set} 
is the subset $\Abn_G\subseteq G$ of all
singular values of the endpoint map. Equivalently, $\Abn_G$ is the union of all {\em abnormal curves} passing through the origin.
If 
the abnormal set
has measure $0$,
then $G$ is said to satisfy the 
{\em Sard Property}.
Proving the Sard Property in the general context of sub-Riemannian manifolds is one of the major open problems in sub-Riemannian geometry, see the questions in \cite[Sec.~10.2]{Montgomery} and Problem III in \cite{Agrachev_problems}.
See also {\cite{Agr_smooth},} \cite{Agrachev_Lerario_Gentile} and \cite{Rifford_Trelat}.
In~\cite{LDMOPV}, the authors of this note and others proved the Sard Property in a number of special cases, {and they also obtained the following first result concerning the interesting problem  of obtaining finer estimates on the size of the abnormal set.}

\begin{theorem*}[{\cite[Theorem 3.15]{LDMOPV}}]
In any free nilpotent group of  step $2$ 
the abnormal set is  an  algebraic subvariety of codimension $3$.
\end{theorem*}

In the present paper we discuss generalizations of the result above to  some classes of step $2$  Carnot groups that are not necessarily free. Our purpose is to present different possible approaches to the problem. As our examples will show, going from the free to the general case does not seem to be a trivial step. {Before summarizing our contributions in the following statement, it is worth recalling that a free-nilpotent group of step 2 and {\em rank} $r$ has dimension $\frac{r^2+r}{2}$}

 \begin{theorem*}
 Let $G$ be a step $2$ Carnot group and let $\dim V_1= r$. The abnormal set $\Abn_G$ is contained in an algebraic subvariety of codimension  three (or more) in the following cases:
 \begin{itemize}
 \item[$(i)$] $G$ has dimension $r+1$ or $r+2$;
 \item[$(ii)$] $G$ has dimension $\frac{r^2+r}{2}-1$ or $\frac{r^2+r}{2}-2$;
 \item[$(iii)$]  $r=4$ and $G$ has dimension $7$.
 \end{itemize}
 \end{theorem*}

After having established the notation and having recalled some known results in Section~\ref{prel}, we state and prove our  contributions. In  Section~\ref{plus1and2} we consider Carnot groups of step $2$ and dimensions $r+1$ and $r+2$. 
This will be the content of Theorem~\ref{plusone} and Theorem~\ref{plustwo}. 
While the case $r+1$ is rather straightforward, in order to prove Theorem~\ref{plustwo} we reason by induction on $r$ and we need some fine observations on the bases of $V_1$. Already in dimension $r+3$  our technique apparently fails to being convenient. In Section~\ref{freeminus},  we consider dimensions $\frac{r^2+r}{2}-1$ and $\frac{r^2+r}{2}-2$ (Theorem~\ref{thm:liberomeno1} and Theorem~\ref{thm:Obiettivo}). 
{The idea here is to see the Carnot groups as quotients of free-nilpotent groups of step $2$ by $1$ and $2-$dimensional ideals, respectively.} Indeed, it turns out (see Proposition~\ref{prop:omomorfismo}) that if $\pi: F\to G$ is a homomorphic projection of stratified groups, then the abnormal curves in $G$ are  the abnormal curves in $F$ that remain abnormal under $\pi$. If the dimension of the kernel of $\pi$ is $1$ or $2$ and $F$ is free, we are able to study the abnormal curves of $G$ using this method. However, when the dimension of the kernel is higher, the computations become  more complicated. Finally, in Section~\ref{4+3} we prove Theorem~\ref{teo:4+3}, which concerns the case where $r=4$ and the dimension is $7$. 

{It can be easily checked that all stratified Lie groups up to dimension $7$ fall  in one of the cases $(i),(ii)$ or $(iii)$ above. In particular,  we have the following consequence.}

 \begin{corollary*}
{Let $G$ be a step $2$ Carnot group of topological dimension not greater than 7; then, the abnormal set $\Abn_G$ is contained in an algebraic subvariety of codimension at most three.}
 \end{corollary*}


\section{Preliminaries}\label{prel}
In this section we establish the notation we are going to use and we recall some preliminary facts that were proved in~\cite{LDMOPV}.

A {\em Carnot} (or {\em stratified}) {\em group} $G$ is a connected, simply connected and nilpotent Lie group whose Lie algebra $\g$ is {\em stratified}, i.e., it has a direct sum decomposition $\g=V_1\oplus\dots\oplus V_s$  such  that
\[
V_{j+1}=[V_j,V_1]\ \forall\: j=1,\dots,s-1,\qquad V_s\neq\{0\}\qquad\text{and}\qquad [V_s,V_1]=\{0\}.
\]
We refer to the integer $s$ as the {\em step} of $G$ and to $r:=\dim V_1$ as its {\em rank}. The group identity will be denoted by $e$.  We will indifferently view $\g$ either as the  tangent space to $G$ at $e$ or as the Lie algebra of left-invariant vector fields in $G$. Recall that in this case the exponential map $\exp:\g\to G$ is a diffeomorphism; we write $\log$ for the inverse of $\exp$. When we  use $\log$  to  identify $\g$ with $G$, the group law  on $G$   becomes
a polynomial map $\g \times \g \to \g$ with $0 \in \g$ playing the role of the identity element $e \in G$.
For all $g\in G$, denote by $L_g$ and $R_g$ the left and right multiplication by $g$, respectively. We write $\Ad_g:= {\rm d}( R_{g^{-1}}\circ L_g)_e:\g\to\g$. For $X\in\g$ we define $\ad_X:\g\to\g$ by $\ad_X(Y):=[X,Y]$.

Fix $u\in L^2([0,1],V_1)$ and denote by $\gamma_u$ the curve  in $G$ solving
\begin{equation}\label{ODE}
\frac{\dd\gamma}{\dd t}(t)=\left(\dd L_{\gamma(t)}\right)_e  u(t), 
\end{equation}
with initial condition $\gamma(0)=e$. Vice versa, if $\gamma:[0,1]\to G$ is an absolutely continuous curve  that solves   \eqref{ODE} for some $u\in L^2([0,1],V_1)$, then we say that $\gamma$ is  {\em horizontal}   and that $u$ is its {\em control}.  The {\em endpoint map} $\End$ is
\begin{eqnarray*}
\End: L^2([0,1],V_1)& \rightarrow &G\\
u\qquad&\mapsto& \gamma_u(1).
\end{eqnarray*}
We will sometimes write $\End^G$ to underline the group $G$ we are working with.

Let $\gamma:[0,1]\to G$ be a horizontal curve   with control $u$ and such that $\ga(0)=e$. If  $ {\rm Im}(\dd \End_u)\subsetneq T_{\gamma(1)} G$ we say that $\gamma$ is {\em abnormal}.
In other words, a horizontal curve $\gamma:[0,1]\to G$ is   abnormal if and only if $ \gamma(1)$ is a critical value of  $\End$. The main goal of this paper is the study of the {\em abnormal set} $\Abn_G$ of $G$ defined by
\begin{equation}\label{abnormal:set}
\Abn_G:=\{ \gamma(1)\, :\, \gamma:[0,1]\to G \text{ abnormal }, \gamma(0)=e\}.
\end{equation}

The following result is proved in~\cite[Proposition 2.3]{LDMOPV}.

\begin{proposition}\label{primo-lemma}
If $\gamma:[0,1]\to G$ is a  horizontal curve leaving from $e$ with control $u $, then 
\begin{equation}\label{for:Image:of:dEnd}	
{\rm Im }(\dd\End_u) =( \dd  R_{\gamma(1)} )_e (\Span\{\Ad_{\gamma(t)}V_1\;:\; t \in [0,1] \}). 
 \end{equation}
\end{proposition}

It is clear from \eqref{for:Image:of:dEnd} that ${\rm Im }(\dd\End_u)$ depends only on $\ga$ and not on its parametrization (i.e., on the control $u$). Given a horizontal curve $\ga:[0,1]\to G$ with $\ga(0)=e$ we define
\begin{equation}\label{eq:defIE}
\IE_\ga:=\Span\{\Ad_{\gamma(t)}V_1\;:\; t \in [0,1] \}.
\end{equation}
By Proposition \ref{primo-lemma}, $\ga$ is abnormal if and only if $\IE_\ga$ is not the whole Lie algebra $\g$. In fact, $\IE_\ga\subset T_eG\equiv\g$ is the image under the diffeomorphism $( \dd  R_{\gamma(1)} )_e^{-1}$ of ${\rm Im }(\dd\End_u)\subset T_{\ga(1)}G$ for any control $u$ associated with $\ga$. Evaluating \eqref{eq:defIE} at $t=0$ and $t=1$ yields
\begin{equation}\label{IER_V_L_V}
V_1 +\Ad_{\ga(1)}V_1\subseteq \IE_\ga.
\end{equation}
We will sometimes use the notation $\IE^G_\ga$  if we need to stress the group under consideration.

\begin{example}\label{ex:Rn}
There are no abnormal curves in $\R^n$ (seen as a step 1 Carnot group). Indeed, by \eqref{IER_V_L_V}, $\g=V_1= \IE_\ga$ for any $\ga$. In particular, $\Abn_{\R^n}=\emptyset$.
\end{example}

We  restrict our analysis to a Carnot group $G$ associated with a stratified Lie algebra $\g=V_1\oplus V_2$ of step 2. Denote by $\pi_{V_1}:\g\to V_1$ the canonical projection; for any fixed $X\in\g$ we recall the formula $\Ad_{\exp(X)}=e^{\ad X}$ to get
\begin{equation}\label{eq:Adstep2}
\Ad_{\exp(X)}(Y) = Y+[X,Y] = Y+[\pi_{V_1}(X),Y] \qquad\forall\: Y\in\g.
\end{equation}
We use this formula to compute more efficiently the linear space $\IE_\ga$ defined in \eqref{eq:defIE}. 

\begin{proposition}\label{prop:abnstep2}
Let $G$ be a Carnot group associated with a stratified Lie algebra $\g=V_1\oplus V_2$ of step 2. Let $\ga$ be a horizontal curve  in $G$ with $\ga(0)=e$ and define the linear space $\Pg\subseteq\g$ by
\begin{equation}\label{eq:defPg}
\Pg:=\Span\{\pi_{V_1}(\log \gamma(t))\;:\; t \in [0,1]\},
\end{equation}
where $\log:G\to\g$ is the inverse of $\exp$. Then 
\[
\IE_\ga=V_1\oplus[\Pg,V_1]
\]
and, in particular, $\ga$ is abnormal if and only if $[\Pg,V_1]\neq V_2$.
\end{proposition}
\begin{proof}
Using \eqref{eq:Adstep2} and the fact that $V_1\subseteq\IE_\ga$ (see \eqref{IER_V_L_V}), we obtain
\[
\begin{split}
\IE_\ga & = \Span\{Y+[\pi_{V_1}(\log \ga(t )),Y]\;:\; t \in [0,1], Y\in V_1 \}\\
& = V_1 \oplus  \Span\{[\pi_{V_1}(\log \ga(t )),Y]\;:\; t \in [0,1], Y\in V_1 \}\\
& = V_1\oplus[\Pg,V_1]
\end{split}
\]
as stated.
\end{proof}

\begin{example}\label{ex:Heisenberg}
Given an integer $n\geq 1$, the $n$-th {\em Heisenberg group} $H^n$ is the Carnot group associated with the $(2n+1)$-dimensional stratified Lie algebra $V_1\oplus V_2$ of step 2 where
\[
V_1:=\Span\{ X_1,\dots,X_n,Y_1,\dots,Y_n\},\qquad V_2:=\Span\{T\}
\]
and the only non-zero commutation relations between the generators are given by
\[
[X_i,Y_i]=T\quad\text{ for any }i=1,\dots,n.
\]
In particular, for any horizontal curve $\ga$ with $\ga(0)=e$ we have $[\Pg,V_1]=V_2$ unless $\Pg=\{0\}$. It follows that the only abnormal curve in $H^n$ is the constant curve $\bar\gamma(t)\equiv e$ and, in particular, $\Abn_{H^n}=\{e\}$. We note here for future reference that  $\IE_{\bar\ga}=V_1$.
\end{example}

Proposition \ref{prop:abnstep2} allows to give a completely algebraic description of $\Abn_G$. Consider an abnormal curve $\ga$ in $G$ leaving from the identity $e$ and let $\Pg$ be as in \eqref{eq:defPg}. Then Im $\ga$ is contained in the subgroup of $G$ associated to the Lie algebra generated by $\Pg$, i.e.,
\[
\text{Im }\ga\subseteq \exp(\Pg\oplus[\Pg,\Pg]).
\]
Assume for a moment that  $\dim \Pg\in \{r,r-1\}$, i.e., that $\Pg$ is either the whole horizontal layer $V_1$ or a hyperplane of $V_1$; in both cases one would have  $[\Pg,V_1]=V_2$ and, by Proposition \ref{prop:abnstep2}, $\ga$ would not be abnormal. This proves that
\[
\Abn_G\subseteq \bigcup\{ \exp(P\oplus[P,P])\,:\, P \text{ linear subspace of }\g,\ \dim P\leq r-2, [P,P]\neq V_2\}.
\]
The reverse inclusion holds as well. Indeed, if $P$ is a linear subspace of $\g$ such that $\dim P\leq r-2$ and $[P,P]\neq V_2$, then  by Chow connectivity theorem any point in the subgroup $H:=\exp(P\oplus[P,P])$ can be connected to $e$ by a horizontal curve $\ga$ entirely contained in $H$, and such a $\ga$ must be abnormal by Proposition \ref{prop:abnstep2}. We have therefore proved the following result (see also \cite[Section 3.1]{LDMOPV}).

\begin{proposition}\label{prop:assumiamor-2}
Let $G$ be a Carnot group associated with a  Lie algebra $\g=V_1\oplus V_2$ of step 2; then
\[
\Abn_G= \bigcup\{ \exp(P\oplus[P,P])\,:\, P \text{ linear subspace of }\g,\ \dim P\leq r-2, [P,P]\neq V_2\}.
\]
\end{proposition}

An immediate consequence of Proposition \ref{prop:assumiamor-2}, proved in~\cite[Theorem 1.4]{LDMOPV}, is the following result.

\begin{theorem}\label{thm:free_step_two}
Let $G$ be a Carnot group associated with a free Lie algebra of step 2; then $\Abn_G$ is contained in an affine algebraic subvariety of codimension $3$.
\end{theorem}

We conclude this section by proving two simple results that hold in general Carnot groups.

\begin{proposition}\label{prop:prodotti}
Let $G$ and $H$ be Carnot groups associated with stratified Lie algebras $\g$ and $\mathfrak h$ (respectively). Let $\ga$ be a horizontal curve in the Carnot group $G\times H$ and write $\ga=(\al,\be)$ for unique horizontal curves $\al$ and $\be$ in $G$ and $H$, respectively. Then, $\ga$ is abnormal in $G\times H$ if and only if either $\al$ is abnormal in $G$ or $\be$ is abnormal in $H$; in particular
\[
\Abn_{G\times H}=(\Abn_G\times H)\cup(G\times\Abn_H).
\]
\end{proposition}
\begin{proof}
Let $V_1$ and $W_1$ be the first layers in the stratifications of $\g$ and $\mathfrak h$ respectively. If $u\in L^2([0,1],V_1)$ and $w\in L^2([0,1],W_1)$ are the  controls associated to $\al$ and $\be$ respectively, then $(u,v)\in L^2([0,1],V_1\times W_1)$ is a control of $\ga$ and 
\[
\dd \End_{(u,v)}^{G\times H}=\dd \End_u^G \otimes \dd \End_v^H.
\]
In particular, $\dd \End_{(u,v)}^{G\times H}$ is surjective if and only if both $\dd \End_u^G$ and $\dd \End_v^H$ are, and this is enough to conclude.
\end{proof}

We need some terminology before stating our next result. Assume that $\mathfrak f=V_1\oplus\dots\oplus V_s$ is a nilpotent stratified Lie algebra and that $\mathfrak I\subset V_2\oplus\dots\oplus V_s$ is an ideal of $\mathfrak f$; consider   the quotient Lie algebra $\g=\mathfrak f/\mathfrak I$. Let $\pi:\mathfrak f\to\g$ be the associated canonical projection and let $F,G$ be the Carnot groups associated with $\mathfrak f, \g$ respectively; we use  the same symbol $\pi:F\to G$ to denote the canonical projection at the group level. It is well-known that any horizontal curve $\ga$ in $G$ leaving from the identity (of $G$) admits a unique {\em lift} to $F$, i.e., a unique horizontal curve $\bar\ga$ in $F$ leaving from the identity (of $F$) and such that $\ga=\pi\circ\bar\ga$. Moreover, essentially by definition (see \eqref{eq:defIE}) we have that for any horizontal curve $c$ in $F$
\begin{equation}\label{eq:proiezIE}
\IE_{\pi\circ c}=\pi(\IE_c).
\end{equation}
This proves the following result.

\begin{proposition}\label{prop:omomorfismo}
Let $\mathfrak f=V_1\oplus\dots\oplus V_s$ be a nilpotent stratified Lie algebra, let $\mathfrak I\subset V_2\oplus\dots\oplus V_s$ be an ideal of $\mathfrak f$ and let $\g=\mathfrak f/\mathfrak I$ be the quotient Lie algebra. Let $\pi:F\to G$ be the canonical projection between the Carnot groups $F$ and $G$ associated with $\mathfrak f$ and $\g$, respectively. Then, for any abnormal curve $\ga$ in $G$, the lift $\bar\ga$ is an abnormal curve in $F$; in particular,
\[
\Abn_G\subset\pi(\Abn_F).
\]
\end{proposition}


\section{Step $2$, rank $r$, dimensions $r+1$ and $r+2$}\label{plus1and2}
In this section we show that if $G$ is a Carnot group of step $2$ such that $\dim V_1=r$ and $\dim V_2 =1$ or $2$, then the abnormal set $\Abn$ has codimension  at least $3$. 
The case where $\dim V_2 =1$ is somewhat elementary.
Given an integer $\ell\geq 0$, we say that a vector $X\in V_1$ has rank $\ell$ if ${\rm rank}(\ad X)=\ell$. We denote by $R_\ell$ the set of vectors in $V_1$ of rank at most $\ell$; notice that $R_0={\mathfrak z}(\mathfrak g)\cap V_1$.

\begin{theorem}\label{plusone}
Let $G$ be the Carnot group associated with a stratified Lie algebra $\g=V_1\oplus V_2$ such that $\dim V_1 =r$ and $\dim V_2 =1$. 
 Then
$\dim \Abn_G \leq r-2$.
\end{theorem}
\begin{proof}
By Proposition \ref{prop:abnstep2},    a curve $\gamma$ is abnormal if and only if
$[\Pg, V_1]=\{0\}$, i.e., 
 $\Pg\subseteq {\mathfrak z}(\mathfrak g)$, the center of $\g$. So the abnormal curves are contained in $\exp R_0$, which is a subgroup of dimension at most $r-2$, for otherwise $\g$ would be abelian. 
\end{proof}

We also provide an alternative proof suggested to us by E. Le Donne.

\begin{proof}[Alternative proof of Theorem \ref{plusone}]
It is an  exercise left to the reader to show that $G$ is isomorphic to the product $H^k\times\R^{h}$ for suitable integers $k\geq 1$ and $h\geq 0$. The conclusion follows from  Proposition \ref{prop:prodotti} and Examples \ref{ex:Rn} and \ref{ex:Heisenberg} .
\end{proof}

The proof of the case $\dim V_2 =2$ is less straightforward, and 
it requires some preliminary results.

\begin{lemma}\label{lemma:rank1codim}
Let $\g=V_1\oplus V_2$ be a step two stratified Lie algebra with $\dim V_1 =r\geq 3$ and $\dim V_2 =2$. 
Then $R_1$ has codimension  at least $1$ in $V_1$.
\end{lemma}

\begin{proof}
Let $X_1,\dots,X_r$ be a basis of $V_1$ and $Z_1,Z_2$ a basis of $V_2$; let $X=\sum_{i=1}^r x_i X_i\in V_1$. Asking $X$ to have rank at most 1 amounts to an algebraic condition on $x_1,\dots,x_r$: in fact, $X$ has rank at most 1 if and only if all the $2\times2$ minors of the matrix $M\in\R^{2\times r}$ defined by $[X_j,X]=M_{1j}Z_1 + M_{2j}Z_2$, $j=1,\dots,r$  have null determinant. We notice that the entries of $M$ are linear in $x_1,\dots,x_r$, hence all the determinants of the $2\times2$ minors are given by a (possibly zero) homogeneous second-order polynomial in $x_1,\dots,x_r$. It follows that $R_1$ is contained in an algebraic variety of $V_1$.

In order to show that this variety is not the whole $V_1$ (and hence it has positive codimension, as desired) it is enough to find a vector in $V_1$ of rank $2$. This is trivial if $r=3$. For $r\geq 4$ we reason by contradiction and assume that all vectors have rank 0 or 1. Fix $X_1,X_2\in V_1$  such that $[X_1,X_2]=Z_1\neq 0$.  We can choose\footnote{{Otherwise $[X_j,X_1]=aZ_1$ and $[X_j,X_2]=bZ_1$ and one could replace $X_j$ with $X_j-bX_1+aX_2$.}\label{foot:banale}} a basis $\{X_1,X_2,X_3,\dots,X_r\}$  of $V_1$ so that $[X_j,X_1]=[X_j,X_2]=0$ for every $j=3,\dots,r$. Since $\dim V_2=2$, we may assume that $Z_1$ and $Z_2:=[X_3,X_4]$ are linearly independent. Then $X_1+X_3$ has rank $2$, and the proof is accomplished.
\end{proof}

\begin{theorem}\label{plustwo}
Let $G$ be the Carnot group associated with a stratified Lie algebra $\g=V_1\oplus V_2$ such that $\dim V_1 =r$ and $\dim V_2 =2$. Then
$\dim \Abn_G\leq r-1$.
\end{theorem}
\begin{proof}
We argue by induction on $r$. Since $\dim V_2 =2$, the smallest possible dimension for $V_1$ is $r=3$. If this is the case, then for every (non-constant) abnormal curve $\gamma$ the space $\Pg$ has dimension one (otherwise $[\Pg,V_1]=V_2$ and $\gamma$ would not be abnormal). Therefore $\gamma$ is a horizontal line, and $\Pg=\Span \{X^\ga\}$ for some $X^\ga \in V_1$; actually, $X^\ga\in R_1$, for otherwise $[\Pg,V_1]=V_2$. It follows that $\Abn_G\subset \exp(R_1)$; by Lemma~\ref{lemma:rank1codim}, this has dimension at most $2$, as stated. 
 
 Assume now that the thesis is true for $\dim V_1\leq r-1$ and let $\gamma$ be an abnormal curve. Notice that if $\dim \Pg =1$, then $\gamma\subseteq \exp (R_1)$. Therefore, if there is no abnormal curve $\gamma$ such that $\dim \Pg\geq 2$, the conclusion follows from Lemma~\ref{lemma:rank1codim}. Otherwise there is a curve $\ga$ which is abnormal (i.e., $\dim [\Pg,V_1]\leq 1$) and such that
$\Pg$ has dimension $\geq 2$. If $\Pg$ is abelian for every such $\ga$, then the abnormal curves are all contained in the space $\exp (R_1)$, which has dimension at most $r-1$ by Lemma~\ref{lemma:rank1codim}.
 
Otherwise, suppose there is an abnormal curve $\ga$ such that $\dim\Pg\geq 2$, $\dim [\Pg,V_1]\leq 1$, and $[\Pg,\Pg]\neq 0$, so that  $\dim [\Pg,\Pg]=\dim[\Pg,V_1]=1$. Then there are $X_1,X_2\in \Pg$ such that $[X_1,X_2]=Z_1\neq 0$. We may complete\footnote{See footnote \ref{foot:banale}.} $X_1$ and $X_2$ to a basis
  $X_1,\dots,X_r$  of $V_1$ so that  
 $[X_j,X_1]=[X_j,X_2]=0$ for every $j=3,\dots,r$. 
 Denoting by $\mathfrak{g}$ the Lie algebra of $G$, we have two cases: either $Z_1$ is in the Lie algebra generated by $X_3,\dots,X_r$, or not.
 In the latter case, we may write  $\mathfrak{g}=\mathfrak{h}\oplus \tilde{\mathfrak{g}}$, a Lie algebra direct sum where $\mathfrak{h}:={\rm span}\{X_1,X_2,Z_1\}$ is the three dimensional Heisenberg Lie algebra, $\tilde{\mathfrak{g}}:={\rm span}\{X_3,\dots,X_r,Z_2\}$ is a subalgebra of $\mathfrak{g}$, and $Z_2$ is a vector in $V_2$ linearly independent from $Z_1$. In particular, by Proposition \ref{prop:prodotti} the abnormal set satisfies 
$$
 \Abn_G\subseteq \left(\Abn_H\times \tilde{G}\right)\cup \left( H\times \Abn_{\tilde{G}}\right)
 $$
 with $H$ and $\tilde{G}$ denoting the subgroups of $G$ with Lie algebras $\mathfrak{h}$ and $\tilde{\mathfrak{g}}$ respectively.  One can then easily conclude using Theorem \ref{plusone} (which applies to $\tilde G$) and Example \ref{ex:Heisenberg}.
 
We are  left with the case where $Z_1$ is in the Lie algebra generated by $X_3,\dots,X_r$.  In this case  $X_3,\dots,X_r$ generate two independent vectors  $Z_1$ and $Z_2$ in $V_2$. Define the Lie algebra $\mathfrak{g}^\prime:={\rm span}\{X_3,\dots,X_r,T,Z_2\}$ with  Lie product given by declaring the map $\phi:\mathfrak{g}\to \mathfrak{g}^\prime$, defined by $\phi(Z_1)=T$ and by the identity on the other basis vectors, to be an isomorphism.  Then we can write 
 $$
 \mathfrak{g}=(\mathfrak h \oplus \g') / \mathfrak I,
 $$
 where  $\mathfrak I:={\rm span}\{T-Z_1\}$ and again $\mathfrak h:={\rm span}\{X_1,X_2,Z_1\}$. 
 The homomorphic surjective projection  $\pi:\mathfrak{h} \oplus \g'\to\g$ defines a natural projection (for which we use the same symbol $\pi$) at the Lie group level 
 $$
 \pi: H\times G' \to G,
 $$
 where $G'$ denotes the exponential group of $\g'$. 
By Proposition \ref{prop:omomorfismo} the abnormal curves $\ga$ in $G$ are projections  of abnormal curves $(\alpha,\beta)$ in $H\times G'$, with $\alpha$ and $\beta$ horizontal curves in $H$ and $G'$ respectively.  Namely, $\gamma=\pi(\alpha,\beta)$. By Proposition \ref{prop:prodotti}, $(\alpha,\beta)$ is abnormal in $H\times G'$ if and only if at least one of $\alpha$ and $\beta$ is abnormal in $H$ and $G'$ respectively. 
So, the abnormal set $\Abn_G$ is contained in ${\rm A}_H \cup {\rm A}_{G'}$, where
$$
{\rm A}_H:=\{\gamma(1) : \gamma=\pi(\alpha,\beta) \text{ abnormal in }G, \gamma(0)=e, \alpha \text{ abnormal in } H\}
$$
and 
$$
{\rm A}_{G'}:=\{\gamma(1) : \gamma=\pi(\alpha,\beta) \text{ abnormal in }G, \gamma(0)=e, \beta \text{ abnormal in } G'\}
$$
We first consider ${\rm A}_H$.
By Example \ref{ex:Heisenberg}, the only abnormal curves in $H$ leaving from $e$ is the constant curve $\bar\alpha\equiv e$; we claim that  $\gamma=\pi(\bar\alpha,\beta)$  is abnormal in $G$ if and only if $\beta$ is abnormal in $G'$. Indeed, if this was not the case, then $\IE_\beta^{G'}=\g'$ and $\IE_{\bar\alpha}^H=V_1^{\mathfrak h}$, where  $V_1^{\mathfrak h}={\rm span}\{X_1,X_2\}$. Recalling \eqref{eq:proiezIE} we would have $\IE_\gamma^G =\pi( \IE_{\bar\alpha}^H\oplus\IE_\beta^{G'})=\g$ and $\gamma$ would not be abnormal in $G$. This proves that $A_H\subseteq\pi(\{e\}\times \Abn_{G'})$ and, in particular, that $A_H$ has codimension at least 5.
 
 Next, we consider ${\rm A}_{G'}$. If $\gamma$ is abnormal in $G$ and $\alpha$ is not constant, then  $\IE_\alpha^H=\mathfrak h$ by Example \ref{ex:Heisenberg}. Hence, in order to have $\IE_\ga^G\neq \g$, it must be that
 $$
\IE^{G'}_\beta \subseteq V_1^{\g'}\oplus \R T,
 $$
 with $V_1^{\g'}={\rm span}\{X_3,\dots,X_r\}$.
 If $\IE^{G'}_\beta=V_1^{\g'}$, then $[P_\beta,V_1^{\g'}]=\{0\}$. So $\beta\subseteq \exp(V_1^{\g'}\cap \mathfrak{z}(\g'))$, with 
 $\mathfrak{z}(\g')$  the center of $\g'$. But $V_1^{\g'}\cap \mathfrak{z}(\g')$ has dimension at most $r-5$, so we reach the conclusion in this case. Otherwise $\beta$ is such that $\IE^{G'}_\beta=V_1^{\g'}\oplus \R T$, which implies that $[P_\beta,V_1^{\g'}]=\R T$. Namely, $\beta \subseteq \exp(M\oplus  \R T)$, where
 $$
 M= \{X\in V_1^{\g'} \,:\, {\rm Im}(\ad X)\subseteq \R T \}.
 $$
Notice that $M$ is a linear subspace of $V_1^{\g'}$ and that $\dim M\leq r-4$. Using the fact that $\pi(T-Z_1)=0$ we obtain
$$
\pi(\alpha,\beta)\subseteq \pi(H\times \exp(M\times\R))=\pi(H\times \exp(M))
$$
and the set on the right hand side has dimension $r-1$. This concludes the proof.

\end{proof}

\section{Step $2$, rank $r$, dimension $\frac{r^2+r}{2}-1$ and $\frac{r^2+r}{2}-2$.}\label{freeminus}
Let $\g=V_1\oplus V_2$ be a step two Lie algebra  with $\dim V_1=r$. Let $\mathfrak f_{r,2}$ be the nilpotent free Lie algebra of rank $r$ and step $2$ and let $F_{r,2}$ be the Carnot group with $\mathfrak f_{r,2}$ as Lie algebra. 
Denote by $W_1\oplus W_2$ a stratification of $\mathfrak f_{r,2}$, and recall that $\dim W_2 = \frac{r(r-1)}{2}$.
Then the Lie algebra $\g$ can be viewed as the quotient of  $\mathfrak f_{r,2}$ by a subspace $W$ of $W_2$. One possible strategy for studying the abnormal set of $G$ is to study those abnormal curves in $F_{r,2}$ that project to abnormal curves on $G$. In the following two sections  we use this idea to study the abnormal set of $G$ when $\dim V_2=\frac{r(r-1)}{2}-1$ (i.e., $\dim W=1$) and $\dim V_2=\frac{r(r-1)}{2}-2$ (i.e., $\dim W=2$). As the dimension of the space $W$  grows, the discussion becomes considerably more complicated and this strategy (apparently) ceases to be convenient.
The following lemma will be useful for both cases
\begin{lemma}\label{lem:insiemeA}
The set $\mathscr A\subseteq F_{r,2}$ defined  by 
\[
\mathscr A:=\bigcup\{\text{Im }\ga:\ga\text{ is a curve in $F_{r,2}$ and }\dim\Pg\leq r-3\}
\]
has dimension $\frac{r^2+r}2 - 6$. 
\end{lemma}
\begin{proof}
Denoting by $Gr(W_1,k)$ the Grassmannian of $k$-planes in $W_1$, we have 
\[
\mathscr A= \bigcup_{h=3}^{r-1}\mathscr A_h,\qquad\text{where}\qquad\mathscr  A_h:=\bigcup_{P\in Gr(W_1,r-h)} \exp(P\oplus[P,P]).
\]
Each set $\mathscr A_h$ can (locally) be  parametrized by a smooth map depending on a number of parameters equal to $\dim Gr(\R^r,r-h)+\dim F_{r-h,2}$, hence it has dimension
\[
h(r-h)+ \frac{(r-h)(r-h+2)}{2} = \frac{r^2+r}2 - \frac{h(h+1)}{2} \leq \frac{r^2+r}2 - 6
\]
because $h\geq 3$.
\end{proof}

\subsection{Dimension $\frac{r^2+r}{2}-1$} We prove the following.

\begin{theorem}\label{thm:liberomeno1}
Let $G$ be the Carnot group associated  with a  step two stratified Lie algebra $\g=V_1\oplus V_2$ with $\dim V_1 =r$ and $\dim V_2 =\frac{r(r-1)}{2}-1$. Then
$\Abn_G$ has codimension at least $3$.
\end{theorem}
\begin{proof}
As we said, the algebra $\g$ is isomorphic to the quotient of  $\mathfrak f_{r,2}$  by a one dimensional subspace of  $W_2$, say $\mathbb R Z$. 
From Proposition~\ref{prop:omomorfismo}, every abnormal curve in $G$ lifts to an abnormal curve in $F_{r,2}$. Denote by $\pi$ the projection of $F_{r,2}$ onto $G$, 
and
let $\gamma$ be an abnormal curve in $F_{r,2}$.  Then $\Pg$ has dimension at most $r-2$ in $W_1$ by Proposition \ref{prop:assumiamor-2}; the set $\Abn_G$ is thus contained in the union $A\cup B$, where
$$
A := \{\pi\circ \gamma (1) : \gamma \text{ abnormal in } F_{r,2} \text{ and }\dim \Pg\leq r-3\}
$$
and 
$$
B := \{\pi\circ \gamma (1) : \gamma \text{ abnormal in } F_{r,2} \text{ and }\dim \Pg= r-2\}.
$$
Since $A$ coincides  with the projection $\pi(\mathscr A)$ of the set $\mathscr A$ introduced in Lemma \ref{lem:insiemeA},   $A$ has codimension at least $ 5$ in $G$. 
Next, we  analyze the set $B$ and consider an abnormal curve $\ga$ in $F_{r,2}$ such that $\Pg$ is $(r-2)$-dimensional; in particular, $\dim[\Pg,W_1]=\frac{r(r-1)}{2}-1=\dim W_2-1$. It follows that the projection $\pi\circ\gamma$ is abnormal if and only if $Z\in [\Pg,W_1]$, which is an easy consequence of the equalities
\[
\IE_{\pi\circ\ga}=\pi(\IE_\gamma)=\pi(W_1\oplus[\Pg,W_1] ),
\]
where we denoted by $\pi$ also the projection $\mathfrak f_{r,2}\to\g$. Let $\ell=\frac{r(r-1)}{2}$ and fix a basis $Z_1,\dots,Z_{\ell}$ of $W_2$. Without loss of generality, we may assume $Z_1=Z$.
 Let $Z_1^\prime,\dots,Z_{\ell-1}^\prime$ be independent vectors of $[\Pg,W_1]$, that we can write as
 $Z_j^\prime=\sum_{i=1}^\ell a_i^j Z_i$, $j=1,\dots,\ell-1$.
 Then $Z\in [\Pg,W_1]$ if and only if the matrix
$$
\begin{pmatrix}
Z&Z_1^\prime &\cdots & Z_{\ell-1}^\prime
\end{pmatrix}
=
\begin{pmatrix}
1 & a_1^1&\dots &a_{\ell-1}^\ell\\
0& \vdots &&\vdots\\
\vdots &\vdots &&\vdots\\
0 & a_\ell^{1}&\dots&a_\ell^{\ell-1}
\end{pmatrix}
$$
has determinant equal to zero. This gives a non-trivial algebraic condition on $\Pg$. Therefore, the space of all $(r-2)$-dimensional subspaces $\Pg$ of $W_1$ such that $Z\in [\Pg,W_1]$ has dimension at most  $\dim Gr(r,r-2) - 1= 2(r-2)-1$.
Now, every $(r-2)$-dimensional plane in $W_1$ generates a Lie algebra of dimension $\frac{(r-2)(r-1)}{2}$. 
So $B$ has dimension at most $2(r-2)-1 + \frac{(r-2)(r-1)}{2}= \frac{r^2+r}{2}-4$, that is, $B$ has codimension at least $ 3$ in $G$. We then conclude that  $\Abn_G$ has codimension  at least $ 3$, as stated.
\end{proof}

\subsection{Dimension $\frac{r^2+r}2-2$}\label{sec:libero-due}

Let $G$ be a Carnot group of step 2 whose algebra $\g=V_1\oplus V_2$ is such that $\dim V_1=r$ and $\dim V_2=\frac{r^2-r}2 - 2$. Let $F=F_{r,2}$ be the free Carnot group of rank $r$ and step 2; in this section it will be convenient to identify the Lie algebra $\f$ of $F$ with $V_1\oplus \Lambda^2 V_1$; accordingly,  we will indifferently use  the symbols $[X,Y]$ and $X\wedge Y$ to denote the Lie bracket of $X,Y\in V_1$. There exists a 2-dimensional subspace $W\subseteq \Lambda^2 V_1$ such that 
\begin{equation}\label{eq:introducoW}
V_2\equiv {\Lambda^2 V_1}/{W}.
\end{equation}
We use the same symbol $\pi$ to denote the canonical projections $\pi:\Lambda^2 V_1\to V_2$ and  $\pi:F\to G$. 

Before stating the main result of this section, some preparatory work is in order. Let $a\in W\setminus\{0\}$ be a 2-vector and let $k\geq 1$ be its rank; let  $e_1,e_2,\dots,e_{2k}\in V_1$ be such that
\[
a=(e_1\wedge e_2) +\dots+(e_{2k-1}\wedge e_{2k}).
\]
We complete $e_1,\dots,e_{2k}$ to a basis $e_1,\dots,e_{r}$ of $V_1$ and we endow $V_1$ of a scalar product making this basis orthonormal. This induces a canonical scalar product on $\Lambda^2 V_1$ making $(e_i\wedge e_j)_{1\leq i<j\leq r}$ an orthonormal basis. Finally, we choose $b\in W$ in such a way that $a,b$ is a basis of $W$; writing $b=\sum_{1\leq i<j\leq r} c_{ij}\ e_i\wedge e_j$, up to replacing $b$ with $b-c_{12}a$ we can assume that 
\begin{equation}\label{eq:introducob}
c_{12}=0.
\end{equation}

Let us introduce the following  skew-symmetric $r\times r$ matrices:
\[
\Om:=\left( 
\begin{array}{ccccccc|c}
0 & 1 &  &  &  & &  & \\
-1 & 0 &  &  &  & &  & \\
 &  & 0 & 1 &  & &  & \\
 &  & -1 & 0 &  & &  & {\bf 0}\\
 &  &  &  & \ddots & &  & \\
 &  &  &  &  & 0 & 1 &  \\
 &  &  &  &  & -1 & 0 &  \\
\hline 
 &  &  & {\bf 0} &  &  &  & {\bf 0} 
\end{array}
\right),
\]
where $k$ blocks of the form ${\tiny{\left(\begin{array}{cc}0&1\\-1&0\end{array}\right)}}$ appear, all not-shown entries are null and ${\bf 0}$ denotes  null matrices of  proper sizes, and
\[
C:=\left( 
\begin{array}{cccccccc}
0 & c_{12} & c_{13} & c_{14} & \cdots \\
-c_{12} & 0 & c_{23} & c_{24} & \cdots \\
-c_{13} & -c_{23} & 0 & c_{34} & \cdots \\
\vdots & \vdots & \vdots & \ddots & \cdots 
\end{array}
\right)
\]
These matrices have the following notable relations with $a,b$. If one considers $x,y\in V_1$ as column-vectors written in the basis $e_1,\dots,e_r$, then
\begin{equation}\label{eq:scalarimatrici}
(x\wedge y)\cdot a =x\cdot(\Om y) = x^t\Om y\quad\text{and}\quad (x\wedge y)\cdot b =x\cdot(C y) = x^tC y,
\end{equation}
where $\cdot$ denotes  scalar product  (either in $V_1$ or in $\Lambda^2 V_1$) and $t$ denotes transposition. Notice also that, by skew-symmetry,
\begin{equation}\label{eq:skews}
x\cdot(\Om y)=-(\Om x)\cdot y\quad\text{and}\quad x\cdot(C y)=-(C x)\cdot y.
\end{equation}

We will later need the following special form of the matrix $C$ in the particular case in which  $\Om$ and $C$ commute.

\begin{lemma}\label{lem:casocommutativo}
Assume that $\Om C= C\Om$. Then, the basis $e_1,\dots, e_r$ of $V_1$ can be chosen in such a way that  
\[
a=(e_1\wedge e_2) +\dots+(e_{2k-1}\wedge e_{2k})
\]
and there exists a $(r-2k)\times(r-2k)$ skew-symmetric matrix $D$ such that
\[
C=\left( 
\begin{array}{ccccccc|c}
0 & c_{12} &  &  &  & &  & \\
-c_{12} & 0 &  &  &  & &  & \\
 &  & 0 & c_{34} &  & &  & \\
 &  & -c_{34} & 0 &  & &  & {\bf 0}\\
 &  &  &  & \ddots & &  & \\
 &  &  &  &  & 0 & c_{2k-1,k} &  \\
 &  &  &  &  & -c_{2k-1,k} & 0 &  \\
\hline 
 &  &  & {\bf 0} &  &  &  & D
\end{array}
\right).
\]
In particular, $c_{ij}=0$ whenever $i=1,\dots,2k-1$ and $j\geq i+2$.
\end{lemma}
\begin{proof}
Since $-\Om^2$ is the projection on span $\{e_1,\dots,e_{2k}\}$, for any $i,j$ such that $i\leq 2k <j$ we have by skew-symmetry
\[
\begin{split}
0 &=  \Om(e_i+e_j) \cdot C\Om (e_i+e_j) = \Om(e_i+e_j) \cdot \Om C(e_i+e_j) \\
&=  - \Om^2 (e_i+e_j) \cdot C(e_i+e_j) = e_i\cdot C(e_i+e_j) = c_{ij}.
\end{split}
\]
In particular, $C$ takes the form
\[
C=\left( 
\begin{array}{c|c}
E & \bf 0\\
\hline 
\bf 0 & D
\end{array}\right)
\]
for suitable $2k\times 2k$ and $(r-2k)\times(r-2k)$ matrices $E,D$.

Let now $i,j\in\{1,\dots, 2k\}$ be fixed with $j$ odd. Upon agreeing that $c_{ii}=0$ and $c_{ij}=-c_{ji}$ if $i>j$, one has
\[
\begin{split}
c_{ij} =&  e_i\cdot Ce_j = e_i\cdot C\Om e_{j+1} \\
=&   e_i\cdot \Om C e_{j+1} = -\Om e_i\cdot  C e_{j+1}=
\left\{
\begin{array}{ll}
-c_{i-1,j+1} & \text{if $i$ is even}\\
-c_{i+1,j+1} & \text{if $i$ is odd.}
\end{array}
\right.
\end{split}
\]
Hence the $2k\times 2k$ matrix $E$ is composed of $k^2$ $(2\times 2)$-blocks of the form ${\tiny{\left(\begin{array}{cc} a&b\\-b&a\end{array}\right)}}$, which in turn gives the following information: upon identifying $\R^{2k}\equiv \C^k$ in the standard way
\[
(x_1,\dots,x_{2k})\longleftrightarrow (x_1+ix_2,\dots, x_{2k-1}+ix_{2k}),
\]
$C$ can be identified with a $\C$-linear endomorphism of $\C^k$ (notice that, with this identification, $\Om$ is the scalar multiplication by $-i$ in $\C^k$). Since $\Om$ and $\C$ commute, each of them preserve the other's eigenspaces; since $\C$ is algebraically closed, they can be simultaneously diagonalized in $\C^k$, which proves the Lemma.
\end{proof}

We will also need the following simple result

\begin{lemma}\label{lem:inunsottospazio}
Let $H$ be an hyperplane in $V_1$; then, the set $B\subseteq F$ defined by
\[
B:=\bigcup_{U\in Gr(H,2)}\exp(U^\perp\oplus [U^\perp,U^\perp])
\]
has dimension at most $\frac{r^2+r}2 - 5$. In particular, the dimension of $\pi(B)\subseteq G$ is at most $\frac{r^2+r}2 - 5$.
\end{lemma}
\begin{proof}
Each set of the form $\exp(U^\perp\oplus [U^\perp,U^\perp])$ is a subgroup of $F$ isomorphic to $F_{r-2,2}$. Hence, the set  $B$ can (locally) be  parametrized by a smooth map of a number of parameters equal to
\[
\dim Gr(H,2) + \dim F_{r-2,2} = 2(\dim H-2) +\frac{(r-2)(r-1)}{2} = \frac{r^2+r}2 - 5.
\]
\end{proof}

We can now state the main result of this section.

\begin{theorem}\label{thm:Obiettivo}
Let $G$ be a Carnot group associated with  a step two stratified Lie algebra $\g=V_1\oplus V_2$ with $\dim V_1 =r$ and $\dim V_2 =\frac{r(r-1)}{2}-2$. Then $\Abn_G$ has codimension at least $3$, i.e., $\dim \Abn_G\leq \frac{r^2+r}2 - 5$.
\end{theorem}
\begin{proof}
As in the previous section, our strategy consists in studying the set of abnormal curves $\gamma$ in $F$ such that $\pi\circ\gamma$ is abnormal in $G$. By Lemma \ref{lem:insiemeA}, if $\ga$ is such that $\dim\Pg\leq r-3$, then $\pi\circ\ga$ is contained in a subset of $G$ with codimension at least 4. It is therefore enough to study those abnormal curves $\ga$ in $F$ such that $\dim\Pg=r-2$ and $\pi\circ\ga$ is abnormal in $G$.

Let then such a $\ga$ be fixed; since $\dim[\Pg,V_1]=\frac{r^2-r}2 -1$, we have $\codim\IE_\ga^F=1$ and, in particular,
\[
\codim\IE_{\pi\circ\ga}^G = \codim\pi(\IE_\ga^F)\leq 1.
\]
Since $\pi\circ\ga$ is abnormal in $G$, we have therefore $\codim\pi(\IE_\ga^F) = 1$, which is equivalent to $W\subseteq \IE_\ga^F=V_1\oplus[\Pg,V_1]$ (recall that $W$ was introduced in \eqref{eq:introducoW}) and, in turn, to 
\begin{equation}\label{eq:(2)}
W\subseteq [\Pg,V_1]=[\Pg^\perp,\Pg^\perp]^\perp.
\end{equation}
We have therefore to study those planes $U:=\Pg^\perp\in Gr(V_1,2)$ such that $W\perp[U,U]$. By Lemma \ref{lem:inunsottospazio} it is enough to study the case in which $U\not\subseteq e_1^\perp$. Since $U$ is 2-dimensional, there exists a unique (up to a sign) unit vector $x\in U\cap e_1^\perp$. Writing $x=(x_1,\dots,x_r)$ with respect to the basis $e_1,\dots,e_r$, we have $x_1^2+\dots+x_r^2=1$ and $x_1=0$. Let then $y\in U$ be the unique (up to a sign) unit vector such that $x\cdot y=0$; notice that $y_1\neq 0$, otherwise $U=$span$\{x,y\}\subseteq e_1^\perp$. Setting
\[
\widetilde M:=\{(x,y)\in\R^r\times\R^r:\|x\|^2=\|y\|^2=1,x_1= 0,x\cdot y=0,y_1\neq 0\},
\]
the map 
\[
\widetilde M\ni (x,y)\longmapsto U_{x,y}:=\text{span}\{x,y\}\in Gr(V_1,2)\setminus Gr(e_1^\perp,2)
\]
is smooth, locally injective and surjective (actually, it is 4-to-1); moreover, $[U_{x,y},U_{x,y}]=$ span $\{x\wedge y\}$. Therefore, the analysis of those planes $U\not\subseteq e_1^\perp$ such that $W\perp[U,U]$ can be reduced to the analysis of the couples $(x,y)\in\widetilde M$ such that $W=$ span $\{a,b\}\perp x\wedge y$, or equivalently (recall \eqref{eq:scalarimatrici}) such that
\[
x\cdot\Om y=x\cdot Cy=0\,.
\]
We divide our study in two cases according to whether $\Om$ and $C$ commute or not.

{\em Case 1: $\Om C\neq C\Om$.} In this case we have ker $(\Om C-C\Om)\neq V_1$, hence ker $(\Om C-C\Om)$ is contained in some hyperplane $H\subseteq V_1$. By Lemma \ref{lem:inunsottospazio} it is enough to consider those couples $(x,y)$ such that $U_{x,y}\not\subseteq H$; notice that this condition implies that either $x\notin$ ker $(\Om C-C\Om)$ or $y\notin$ ker $(\Om C-C\Om)$. We will show that the set
\begin{equation}\label{eq:defM}
M:=\left\{(x,y)\in\R^r\times\R^r:
\begin{array}{l}
 \|x\|^2=\|y\|^2=1,x_1= 0,x\cdot y=0,y_1\neq 0,\\
 x\cdot\Om y=x\cdot Cy=0,\\
\|(\Om C-C\Om)x\|^2 + \|(\Om C-C\Om)y\|^2> 0
\end{array}
\right\}
\end{equation}
is the union of two smooth manifolds $M_1,M_2$ of dimension $2r-6$, and we claim that this will be enough to prove Theorem \ref{thm:Obiettivo} in  Case 1. Indeed, by the discussion above, in order to prove that $\Abn_G$ has dimension at most  $\frac{r^2+r}2 - 5$ it is enough to show that (the projection on $G$ of) the set
\begin{equation}\label{eq:uuunione}
\bigcup_{(x,y)\in M}\exp(U_{x,y}^\perp\oplus [U_{x,y}^\perp,U_{x,y}^\perp])\ \cup\ \bigcup_{U\in Gr(e_1^\perp,2)\cup Gr(H,2)}\exp(U^\perp\oplus [U^\perp,U^\perp])
\end{equation}
has dimension at most $\frac{r^2+r}2 - 5$. This would be true because the second set in the right hand side has dimension at most $\frac{r^2+r}2 - 5$ by Lemma \ref{lem:inunsottospazio}, while (provided $M=M_1\cup M_2$ as claimed above) the first set can (locally) be parametrized by two smooth maps of $ (2r-6) + \dim F_{r-2,2}= \frac{r^2+r}2 - 5$ parameters.

We start by showing that
\[
M_1:=M\cap\{(x,y)\in\R^r\times\R^r:y\cdot(\Om C-C\Om)x\neq 0\}
\]
is a smooth submanifold of dimension $2r-6$. Defining the open set
\[
\mathcal O_1:= \{(x,y)\in\R^r\times\R^r:y_1\neq 0, \|(\Om C-C\Om)x\|^2 + \|(\Om C-C\Om)y\|^2> 0, y\cdot(\Om C-C\Om)x\neq 0\}
\]
we have 
\[
M_1=\mathcal O_1\cap \left\{(x,y)\in\R^r\times\R^r:F(x,y)=(1,1,0,0,0,0)
\right\}
\]
where $F(x,y):=(\|x\|^2,\|y\|^2,x\cdot y,x_1,x\cdot\Om y,x\cdot Cy)$. We have to show that the rank of $\nabla F$ is 6 at all points of $M_1$; using \eqref{eq:skews} one can compute 
\[
\nabla F(x,y)=\left(
\begin{array}{c|c|c|c|c|c}
2x & {\bf 0} & y & \ep_1 & \Om y & Cy \\
\hline
{\bf 0} & 2y & x & {\bf 0} & -\Om x & -Cx 
\end{array}
\right),
\]
where $\ep_1=(1,0,\dots,0)^t\in\R^r$ and ${\bf 0}$ is the null (column) vector in $\R^r$. Assume by contradiction that there exists $(x,y)\in M_1$ such that the six columns are linearly dependent, i.e., there exists $\la\in \R^6\setminus\{0\}$ such that
\begin{equation}\label{eq:(4)}
\begin{cases}
\la_1 x + \la_3 y+\la_4 \ep_1 + \la_5 \Om y+\la_6 Cy={\bf 0}\\
\la_2y + \la_3x- \la_5 \Om x-\la_6 Cx={\bf 0}
\end{cases}
\end{equation}
Taking into account \eqref{eq:skews} and the fact that $(x,y)\in M$, by scalar multiplying by $x$ and $y$ the two equalities in \eqref{eq:(4)}  one obtains
\[
\begin{cases}
\la_1  +\la_4 x_1=0\\
\la_3=0\\
\la_3 +\la_4 y_1=0\\
\la_2=0,
\end{cases}
\]
i.e., $\la_1=\la_2=\la_3=\la_4=0$, because $x_1=0$ and $y_1\neq 0$. Therefore $(\la_5,\la_6)\neq(0,0)$ are such that
\[
\begin{cases}
\la_5 \Om y+\la_6 Cy={\bf 0}\\
\la_5 \Om x+\la_6 Cx={\bf 0},
\end{cases}
\]
i.e., $(\Om x,\Om y)$ and $(Cx,Cy)$ are linearly dependent, hence
\[
0=-\Om y\cdot Cx + Cy\cdot\Om x =y\cdot(\Om C-C\Om)x,
\]
a contradiction.

We now show that $M_2:=M\setminus M_1$ is a $(2r-6)$-dimensional smooth manifold. 
Defining the open set
\[
\mathcal O_2:= \{(x,y)\in\R^r\times\R^r:y_1\neq 0, \|(\Om C-C\Om)x\|^2 + \|(\Om C-C\Om)y\|^2> 0\}
\]
we have 
\[
M_2=\mathcal O_2\cap \left\{(x,y)\in\R^r\times\R^r:G(x,y)=(1,1,0,0,0,0,0)
\right\}
\]
where $G(x,y):=(\|x\|^2,\|y\|^2,x\cdot y,x_1,x\cdot\Om y,x\cdot Cy,y\cdot (\Om C-C\Om)x)$. We have to show that the rank of $\nabla G$ is 6 at all points of $M_2$; using $y\cdot (\Om C-C\Om)x=-x\cdot (\Om C-C\Om)y$ one can compute 
\[
\nabla G(x,y)=\left(
\begin{array}{c|c|c|c|c|c|c}
2x & {\bf 0} & y & \ep_1 & \Om y & Cy & -(\Om C-C\Om)y\\
\hline
{\bf 0} & 2y & x & {\bf 0} & -\Om x & -Cx & (\Om C-C\Om)x
\end{array}
\right).
\]
Assume by contradiction that there exists $(x,y)\in M_2$ such that the rank of $\nabla G(x,y)$ is not 6; then, the first six columns are linearly dependent and, reasoning as above, one can show the existence of $(\la_5,\la_6)\neq(0,0)$  such that
\[
\la_5 (\Om x,\Om y)+\la_6 (Cx, Cy)=0.
\]
Since $y_1\neq 0$ we have $\Om y\neq 0$, hence $\la_6\neq0$; in particular
\begin{equation}\label{eq:la}
(Cx, Cy) = \la (\Om x,\Om y)\qquad\text{for }\la:=-\tfrac{\la_5}{\la_6}.
\end{equation}
Since also the columns 1--5 and 7 of $\nabla G$ are linearly dependent, there exists $\mu\in \R^6\setminus \{0\}$ such that
\[
\begin{cases}
\mu_1 x + \mu_3 y+\mu_4 \ep_1 + \mu_5 \Om y-\mu_6 (\Om C-C\Om)y={\bf 0}\\
\mu_2y + \mu_3 x- \mu_5 \Om x+\mu_6 (\Om C-C\Om)x={\bf 0}.
\end{cases}
\]
After scalar multiplication by $x$ and $y$, and taking into account that $v\cdot(\Om C-C\Om)v=0$ for any $v\in V_1$, one gets
\[
\begin{cases}
\mu_1  +\mu_4 x_1=0\\
\mu_3=0\\
\mu_3 +\mu_4 y_1=0\\
\mu_2=0,
\end{cases}
\]
and, as before, $\mu_1=\mu_2=\mu_3=\mu_4=0$. We deduce that there exist $(\mu_5,\mu_6)\neq(0,0)$ such that $\mu_5\Om y -\mu_6 (\Om C-C\Om)y={\bf 0}$. We scalar multiply this equation by $\Om y $ (which is not null because $y_1\neq 0$) to get
\begin{eqnarray*}
0 & =& \mu_5\|\Om y\|^2 - \mu_6 \Om y\cdot (\Om C-C\Om)y\\
& \stackrel{\eqref{eq:la}}{=} &\mu_5\|\Om y\|^2 - \mu_6 \big(\Om y\cdot \la\Om\Om y - \Om y\cdot C\Om y \big) =\mu_5\|\Om y\|^2,
\end{eqnarray*}
which contradicts the fact that $\Om y\neq 0$. This concludes Case 1.

{\em Case 2: $\Om C= C\Om$.} Let $C$ be as in Lemma \ref{lem:casocommutativo}; we can also assume \eqref{eq:introducob}. Reasoning as before (and, in particular, using Lemma \ref{lem:inunsottospazio} and the fact that $C\neq0$, i.e., that ker $C$ is contained in some hyperplane of $\R^r$), it will be enough to show that the variety 
\[
\begin{split}
N:= & \left\{
(x,y)\in\R^r\times\R^r:
\begin{array}{l}
 \|x\|^2=\|y\|^2=1,x_1=0,x\cdot y=0,y_1\neq 0,\\
 x\cdot(\Om y)=x\cdot (Cy)=0\text{ and } x,y\notin\text{ker }C
\end{array}
\right\}
\\
= & \{(x,y)\in\R^r\times\R^r: y_1\neq 0,x,y\notin\text{ker }C\text{ and }F(x,y)=(1,1,0,0,0,0)\}
\end{split}
\]
(where $F$ is as above) is $(2r-6)$-dimensional, i.e., that for any $(x,y)\in N$ the columns of $\nabla F(x,y)$ are linearly independent. If this were not the case, reasoning as in Case 1 one would find $(x,y)\in N$ and $(\la_5,\la_6)\neq(0,0)$ such that 
\[
\la_5 \Om x+\la_6 Cx={\bf 0}\quad\text{and}\quad
\la_5 \Om y+\la_6 Cy={\bf 0}.
\]
By the particular forms of the matrices $\Om$ and $C$ one has
\[
0=(\la_5 \Om y+\la_6 Cy)\cdot (0,1,0,\dots,0) = -\la_5 y_1
\]
and, since $y_1\neq0$, we get $\la_5=0$ and $\la_6\neq 0$. This implies that $x,y\in $ ker $C\neq\R^r$, a contradiction.
\end{proof}


\section{Step $2$, rank $4$, dimension $7$}\label{4+3}
The cases discussed in the previous sections cover all Carnot groups of dimension up to 6 and, among Carnot groups of dimension 7, only those with rank 4 are left out. We discuss them in the following Theorem \ref{teo:4+3}

\begin{lemma}\label{lem:basespeciale4+3}
Let $\g=V_1\oplus V_2$ be a  step two stratified Lie algebra   with $\dim V_1=4$ and $\dim V_2=3$; assume that there exists a linear subspace $P\subset V_1$ such that
\[
\text{$\dim P=2$,\quad $\dim[P,V_1]=2$\quad and\quad $\dim [P,P]=1$.}
\]
Then, there exist $\la\in\R$ and bases  $X_1,\dots,X_4$ of $V_1$ and $X_{21},X_{31},X_{32}$ of $V_2$ such that
\begin{align*}
& [X_2,X_1]=X_{21}, \quad [X_3,X_1]=X_{31}, \quad [X_4,X_3]=X_{43},\quad [X_4,X_2]=\lambda X_{31},\\
& [X_4,X_1]=[X_3,X_2]=0.
\end{align*}
\end{lemma}
\begin{proof}
We can first fix a basis $X_1,X_2$ of $P$ and a vector $X_3\in V_1\setminus P$ such that $X_{21}:=[X_2,X_1]$ and $X_{31}:=[X_3,X_1]$ are linearly independent. By assumption we will have
\[
[X_3,X_2]=aX_{21}+bX_{31}
\]
and, up to replacing $X_3,X_2$ with (respectively) $X_3+aX_1,X_2-bX_1$, we can assume $a=b=0$. Complete $X_1,X_2,X_3$ to a basis of $V_1$ by choosing some $X_4\in V_1$; we will have
\[
[X_4,X_1]=cX_{21}+dX_{31}
\]
and we can assume $c=d=0$ up to replacing $X_4$ with $X_4-cX_2-dX_3$. At this point, we have
\[
[X_4,X_2]=eX_{21}+\lambda X_{31}
\]
and we can assume $e=0$ up to replacing $X_4$ with $X_4+eX_1$. The choice of $X_{43}:=[X_4,X_3]$, which must necessarily be independent from $X_{21},X_{31}$, completes the proof.
\end{proof}

\begin{theorem}\label{teo:4+3}
Let $G$ be a Carnot group associated to a  step two stratified Lie algebra $\g=V_1\oplus V_2$  with $\dim V_1=4$ and $\dim V_2=3$. Then $\Abn_G$ has codimension at least $ 3$, i.e., $\dim \Abn_G\leq 4$.
\end{theorem}
\begin{proof}
By Proposition \ref{prop:assumiamor-2}, $\dim\Pg\leq 2$ for any abnormal curve $\ga$ in $G$. The union 
\[
\bigcup\Big\{{\rm Im\ }\ga:\ga\text{ abnormal in $G$ and }\dim\Pg=1\Big\}
\]
is contained in $\exp (V_1)$, whose codimension is 3. Similarly, if $\Pg$ has dimension 2 but it is Abelian, then is contained in $\exp (V_1)$.

We then have to consider only those abnormal curves $\ga$ such that $\dim\Pg=2$, $[\Pg,V_1]\neq V_2$ and $\dim [\Pg,\Pg]=1$. If  there existed one such $\ga$ with $\dim [\Pg,V_1]=1$, one could choose a basis $X_1,\dots,X_4$ of $V_1$ with $X_1,X_2\in \Pg$  such that 
$$
[X_i,X_j]=
\begin{cases}
X_{21}\neq 0  &\text{if } (i,j)=(2,1)\\
\in \R X_{21} & \text{for } i=1,2 \text{ and } j=3,4,
\end{cases}
$$
which would contradict the fact that $\dim V_2 =3$. Therefore all abnormal curves $\ga$ we have to consider satisfy
\begin{equation}\label{eq:suchagamma}
\text{$\dim\Pg=2$,\quad $\dim[\Pg,V_1]=2$\quad and\quad $\dim [\Pg,\Pg]=1$.}
\end{equation} 
Assuming that one such $\ga$ exists (otherwise the proof would be concluded) we can use  Lemma \ref{lem:basespeciale4+3} to find bases  $X_1,\dots,X_4$ of $V_1$ and $X_{21},X_{31},X_{32}$ of $V_2$ such that 
\begin{align*}
& [X_2,X_1]=X_{21}, \quad [X_3,X_1]=X_{31}, \quad [X_4,X_3]=X_{43},\quad [X_4,X_2]=\lambda X_{31},\\
& [X_4,X_1]=[X_3,X_2]=0
\end{align*}
for a suitable  $\lambda\in\R$. Let $\ga$ be an abnormal curve as in \eqref{eq:suchagamma}; the end point of $\gamma$ will be the exponential $\exp(U)$ of a vector
$$
U=X+Z= \sum_{i=1}^4 x_i X_i + x_{21}X_{21}+x_{31}X_{31} + x_{43}X_{43}
$$
with $X=\sum_{i=1}^4 x_i X_i \in \Pg$. We can safely assume that $X\neq 0$, for otherwise the endpoint $\exp(U)$ of $\ga$ would be contained in a variety of dimension 3.  Since $\dim [X,V_1] \leq 2$, we have to require that the matrix
$$
M(X):=
\begin{pmatrix}
x_2 & -x_1&0&0\\
x_3&\lambda x_4&-x_1&-\lambda x_2\\
0&0&x_4&-x_3
\end{pmatrix},
$$
has rank not greater than 2, where the  columns of $M(X)$ represent (in the basis $X_{21}, X_{31},X_{43}$) the vectors $[X,X_i]$ for $i=1,2,3,4$. On computing the determinants of the four $3\times 3$ minors we obtain
$$
x_i(x_1x_3+\lambda x_2x_4)=0 \quad\forall i=1,\dots,4
$$
and, since $X\neq 0$, we deduce that
\begin{equation}\label{eq:luogodizeri}
x_1x_3+\lambda x_2x_4=0.
\end{equation}
We also notice that, since $Z\in[\Pg,\Pg]\subset {\rm Im}\ad_X$, $Z$ must be a linear combination of the columns of $M(X)$.

From now on we use exponential coordinates adapted to the basis $X_1,\dots,X_{43}$ of $\g$, i.e., we identify $G$ and $\R^7$ by
\begin{multline*}
\R^7\ni(x_1,x_2,x_3,x_4,x_{21},x_{31},x_{43})\\
\longleftrightarrow\exp(x_1X_1+x_2X_2+x_3X_3+x_4X_4+ x_{21}X_{21}+x_{31}X_{31}+x_{43}X_{43})\in G.
\end{multline*}

Assume first that either $x_2=x_3=0$ or $x_1=x_4=0$; then the endpoint $\exp(U)$ of $\gamma$ belongs to the variety $  A \cup B,$ where 
$$
A:=\{(x_1,0,0,x_4,x_{21},x_{31},x_{43})\, :\, (x_{21},x_{31},x_{43})=a(-x_1,\lambda x_4,0)+b(0,-x_1,x_4), x_1,x_4,a,b\in \R\}
$$
and 
$$
B=\{(0,x_2,x_3,0,x_{21},x_{31},x_{43})\, :\, (x_{21},x_{31},x_{43})=a(x_2,x_3,0)+b(0,-\lambda x_2,-x_3), x_2,x_3,a,b\in \R\}.
$$
This easily follows from the fact that $Z$ is a linear combination of the columns of $M(X)$. The variety $A\cup B$ has dimension 4.

We can now assume that 
\begin{equation}\label{eq:acoppie}
(x_1,x_4)\neq (0,0)\qquad\text{and}\qquad(x_2,x_3)\neq (0,0).
\end{equation}
Consider $Y=\sum_{i=1}^4 y_iX_i\in\Pg$ linearly independent from $X$: without loss of generality, we can also assume\footnote{Up to replacing $Y$ with $Y+cX$ for a proper $c\in\R$.} that if $x_i\neq 0$ then $y_i\neq 0$. We consider three cases to conclude the proof.
\begin{itemize}
\item[(i)] Suppose  that $x_2\neq 0$ and $x_4\neq 0$. Then ${\rm Im} \ad_X$ contains the linearly independent vectors $(x_2,x_3,0)$ and $(0,-x_1,x_4)$. Since $[\Pg,V_1]$ has dimension two, ${\rm Im} \ad_Y$ is also linearly generated by $(x_2,x_3,0)$ and $(0,-x_1,x_4)$. In particular, our choice of $Y$ yields
$$
(y_2,y_3,0)=a(x_2,x_3,0) \quad \text{and}\quad (0,-y_1,y_4)=b(0,-x_1,x_4),
$$
for $a,b\in \R\setminus\{0\}$. Since $Z$ is a multiple of
\begin{align*}
[X,Y]&=(b-a)\big(x_1x_2X_{21}+(x_1x_3-\lambda x_2x_4)X_{31}-x_3x_4X_{43}\big)\\
& =(b-a)\big(x_1x_2X_{21}+2x_1x_3X_{31}-x_3x_4X_{43}\big)
\end{align*}
by \eqref{eq:luogodizeri}, the endpoint $\exp(U)$ of $\gamma$ is in the four dimensional variety
$$
\{(x_1,x_2,x_3,x_4,tx_1x_2,2tx_1x_3,-tx_3x_4)\,:\,x_1,x_2,x_3,x_4,t\in\R,x_1x_3+\lambda x_2x_4=0\}.
$$
\item[(ii)] Suppose that $x_2=0$. By \eqref{eq:acoppie}, $x_3\neq 0$ and $M(x)$ has two linear independent columns $(0,x_3,0)$ and $(0,0,x_3)$, which therefore generate ${\rm Im} \ad_X$. By \eqref{eq:luogodizeri}, $x_1=0$. Then the endpoint $\exp(U)$ of $\gamma$ is in the four dimensional variety 
$$
\{(0,0,x_3,x_4,0,x_{31},x_{43})\,:\, x_3,x_4,x_{31},x_{43}\in\R\},
$$
because $(x_{21},x_{31},x_{43})\in {\rm Im}\ad_X$.
\item[(iii)] Suppose that $x_4=0$. By \eqref{eq:acoppie}, $x_1\neq 0$ and the vectors $(1,0,0)$ and $(0,1,0)$ generate the column space of $M(x)$. By \eqref{eq:luogodizeri}, $x_3=0$. As before, we conclude that the endpoint $\exp(U)$ of $\gamma$  is in
$$
\{(x_1,x_2,0,0,x_{21},x_{31},0)\,:\,x_1,x_2,x_{21},x_{31}\in\R\},
$$
that has dimension 4.
\end{itemize}
The proof is accomplished.
\end{proof}


\bibliographystyle{acm}


\end{document}